\documentclass[twoside, 11pt]{article}
\usepackage[english]{babel}
\usepackage[latin1]{inputenc}
\usepackage{amsmath,amssymb}
\usepackage{amsfonts}
\usepackage{latexsym}
\usepackage{graphics,graphicx,psfrag}
\usepackage{upref}
\usepackage{hyperref}
\usepackage{caption,accents, float}
\usepackage{cite}
\usepackage[dvipsnames]{color, xcolor}
\usepackage{changes}

\allowdisplaybreaks[1]

\numberwithin{equation}{section}

\newcommand{\R}{\mathbb{R}}

\newcommand{\supp}{\mathop{\mathrm{supp}\,}\nolimits}

\newcommand{\ds}{\displaystyle}
\newcommand{\real}{\mathbb{R}}
\newcommand{\rn}{\mathbb{R}^n}
\newcommand{\nat}{\mathbb{N}}

\newcommand{\rd}{\mathbb{R}^d}
\newcommand{\ud}{\mathrm{d}}

\newtheorem{lemma}{Lemma}[section]

\newtheorem{theorem}[lemma]{Theorem}
\newtheorem{corollary}[lemma]{Corollary}
\newtheorem{definition}[lemma]{Definition}

\newtheorem{rem}[lemma]{Remark}
\newcommand{\remark}[1]{\begin{rem}{\upshape #1}\end{rem}}
\newtheorem{example}[lemma]{Example}

\newcommand{\proofstart}{\mbox{P\,r\,o\,o\,f\, :\quad}}
\newcommand{\proofend}{\nopagebreak\hfill\raisebox{0.3em}{\fbox{}}\\}
\newenvironment{proof}{\proofstart}{\proofend}

\newcommand{\domain}{\mathcal{O}}

\newcommand{\bit}{\begin{itemize}}
\newcommand{\eit}{\end{itemize}}
\newcommand{\beq}{\begin{equation}}
\newcommand{\eeq}{\end{equation}}

\newcommand{\bfa}{\mathbf{a}}
\newcommand{\bfb}{\mathbf{b}}
\newcommand{\bfr}{\mathbf{r}}
\newcommand{\bfs}{\mathbf{s}}
\newcommand{\bfm}{\mathbf{m}}
\newcommand{\bfl}{\boldsymbol\ell}
\newcommand{\bfalpha}{\boldsymbol{\alpha}}

\title{Anisotropic Besov regularity of parabolic PDEs}

\author{
Stephan Dahlke\footnote{Philipps-University Marburg,  FB12 Mathematics and Computer Science, Hans-Meerwein Stra\ss{}e, Lahnberge, 35032 Marburg, Germany. Email: \href{mailto:dahlke@mathematik.uni-marburg.de}{dahlke@mathematik.uni-marburg.de}}\ \thanks{The work of this author has been supported by Deutsche Forschungsgemeinschaft (DFG), Grant No. DA 360/22-1.}
\quad and \quad 
Cornelia Schneider\footnote{\emph{Corresponding author}. Friedrich-Alexander University Erlangen-Nuremberg, Applied Mathematics III, Cauerstr. 11, 91058 Erlangen, Germany. Email: \href{mailto:cornelia.schneider@math.fau.de}{cornelia.schneider@math.fau.de}}\ \thanks{The work of this author has been supported by Deutsche Forschungsgemeinschaft (DFG), Grant No. SCHN 1509/1-2.\hfill 
}}

\date{\today}

\begin{document}

\maketitle

\begin{abstract}
This paper is concerned with the regularity of  solutions to parabolic evolution equations. 
%We  consider right hand sides which allow us to study even nonlinear problems on non-convex domains with the help of Schauder's fixed point theorem -- in contrast to our findings in \cite{DS20}.  
%In contrast to \cite{DS20} we now  consider right hand sides which allow to prove existence and regularity of a solution for nonlinear problems on non-convex domains with the help of Schauder's fixed point theorem. 
%extending our findings in \cite[Thms.~4.5,4.9,4.12,4.14]{DS19} to  domains of polyhedral type.   
Special attention is paid to the smoothness in the specific anisotropic scale 
$\ B^{r\bfa}_{\tau,\tau}, \ \frac{1}{\tau}=\frac{r}{d}+\frac{1}{p}\ $ of Besov spaces where $\bfa$ measures the anisotropy.   The regularity in these spaces determines the approximation order
that can be achieved by fully space-time adaptive approximation schemes. In particular, we show that for the heat equation our results significantly improve \cite{AG12}. \\[2ex]
{{\em Math Subject Classifications. Primary:}   35B65,  46E35.   {\em Secondary:}  35K05, 65M12. }\\[1ex]
%\footnotetext{\textit{Math Subject Classifications.} 35J06, 46E06, 47B06.}

\end{abstract}

\tableofcontents

\section{Introduction}

This paper is concerned with the study of anisotropic Besov regularity estimates for parabolic partial diffenential equations (PDEs). Regularity estimates in Besov spaces are always important since they determine the convergence order of adaptive and other nonlinear constructive approximation schemes for the corresponding, unknown solution. In contrast to this, it is the classical Sobolev smoothness that determines the convergence order of more conventional, uniform schemes. We refer e.g. to DeVore \cite{DV98} and \cite{DDV97}. For elliptic PDEs, a lot of results in this direction have been achieved  in recent years, see   \cite{DDV97,DHS17,DS09} and many others. In all these cases, the Besov smoothness was generically higher than the Sobolev regularity which justifies the use of adaptive algorithms. However, most of these results are concerned with {\em isotropic} Besov estimates which fit perfectly to the stationary character of elliptic partial differential equations. Quite recently, also Besov regularity results for parabolic PDEs have been investigated, see e.g. \cite{DS17, Sch21}. In these works, the authors derived time-dependent Besov regularity in space. In particular, they determine the convergence order of space-adaptive numerical schemes  such as classical time-marching schemes for parabolic equations. For good reasons, in recent years the development of numerical schemes working on the whole space-time cylinder has become more and more important   \cite{SS09}. In many cases, these schemes are simply more efficient. However, if we take the whole space-time cylinder into account, then anisotropic structures occur, since we have (in the simplest case) one derivative in time but two derivatives in space. Therefore, anisotropic Besov spaces might be reasonable choices for the regularity spaces. First results for the heat equation have been obtained by Aimar and Gomez \cite{AG12}. The results in this very interesting paper rely on a certain interpolation technique  in scales with $p>1$, which naturally limits the applicability of their approach. Therefore, in this paper, we follow a different line:  In the meantime, it has turned out that a very efficient way to establish Besov regularity for the solution to a  PDE is first to study the regularity in weighted Sobolev spaces, the  so-called Kondratiev spaces \cite{DHS17a}. The reason is that very sharp embeddings of Kondratiev spaces into Besov spaces habe been  derived. The  whole program has, e.g., very efficiently been carried out in \cite{DHS17, Sch21}. Usually,  Kondratiev spaces can be used for a very precise description of the singularities of the solutions. For our purposes, clearly anisotropic Kondratiev spaces are needed. Therefore, in our setting, the following tasks have to be solved:
\begin{itemize}
\item Define suitable anisotropic Kondratiev spaces and establish embeddings into anisotropic Besov spaces.
\item Establish anisotropic Kondratiev regularity for the problem under
consideration.
\end{itemize}
In our case, we define the anisotropic Kondratiev spaces simply by means of anisotropic weigths, whereas the anisotropic Besov spaces are defined by tensor products of differently scaled wavelets, where the scaling is compatible with the anisotropy. In this setting, the desired embedding is possible. Moreover, we show that for the heat equation the regularity problem in anisotropic Kondratiev spaces is solvable. % and  Kondratiev regularity of the data more or less immediately carries over to the solution. 
Combining these facts yields our main result. \\
This paper is organized as follows: In Section 2 we recall the notation used throughout the paper. Section 3 is dedicated to anisotropic function spaces and their relations. In particular, we deal with anisotropic Sobolev and Besov spaces. Moreover, we introduce anisotropic Kondratiev spaces and study in Section 4 their relations with anisotropic Besov spaces via embeddings. Finally, in Section 5 we use our obtained results in order to investigate the regularity of solutions of the heat equation in anisotropic Besov spaces and compare the outcome with the results from  \cite{AG12}.

\section{Preliminaries}
\label{app-not}

We collect some  notation used throughout the paper. As usual,  we denote by $\nat$ the set of all natural numbers, $\nat_0=\mathbb N\cup\{0\}$, and 
$\real^d$, $d\in\nat$,  the $d$-dimensional real Euclidean space with $|x|$, for $x\in\real^d$, denoting the Euclidean norm of $x$. 
By $\mathbb{Z}^d$ we denote the lattice of all points in $\real^d$ with integer components. 
For $a\in\real$, let  
$\lfloor a\rfloor $ denote its integer part and $a_+:=\max(a,0)$. \\
Moreover,  $c$ stands for a generic positive constant which is independent of the main parameters, but its value may change from line to line. 
The expression $A\lesssim B$ means that $ A \leq c\,B$. If $A \lesssim
B$ and $B\lesssim A$, then we write $A \sim B$.  

Given two quasi-Banach spaces $X$ and $Y$, we write $X\hookrightarrow Y$ if $X\subset Y$ and the natural embedding is bounded. 
By $\supp f$ we denote the support of the function $f$. 
%For a domain $\Omega\subset \real^d$ and $r\in \nat\cup \{\infty\}$ we write $C^r(\Omega)$ for the space of all {real}-valued $r$-times continuously differentiable functions, 
%whereas $C(\Omega)$ is the space of bounded uniformly continuous functions, and  $\mathcal{D}(\Omega)$ for the set of test functions, i.e., the collection of all infinitely differentiable functions with  {compact support contained in $\Omega$. Moreover,  $L^1_{\text{loc}}(\Omega)$ denotes the space of locally integrable functions on $\Omega$.} \\
%For  a multi-index  $\alpha = (\alpha_1, \ldots,\alpha_d)\in \nat_0^d$ with  $|\alpha| := \alpha_1+\ldots+ \alpha_d=r$, $r\in \nat_0$,  and an $r$-times differentiable function $u:\Omega\rightarrow \real$, we write 
%\[
%D^{(\alpha)}u=\frac{\partial^{|\alpha|}}{\partial x_1^{\alpha_1}\dots \partial x_d^{\alpha_d}} u
%\]
%for the corresponding classical partial derivative as well as $u^{(k)}:=D^{(k)}u$ in the one-dimensional case. Hence, the space $C^r(\Omega)$ is normed by 
%\[
%\| u| C^r(\Omega)\|:=\max_{|\alpha|\leq r}\sup_{x\in \Omega}|D^{(\alpha)}u(x)|<\infty. 
%\]
Moreover, $\mathcal{S}(\real^d)$ denotes the Schwartz space of rapidly decreasing functions. The set of distributions on $\Omega$ will be denoted by $\mathcal{D}'(\Omega)$, whereas $\mathcal{S}'(\real^d)$ denotes the set of tempered distributions on $\real^d$. The terms {\em distribution} and {\em generalized function} will be used synonymously. Furthermore, let $\hat{f}$ stand for the Fourier transform on $\mathcal{S}'(\real^d)$ with inverse ${f}^\vee$. 

 For the application of a distribution $u\in \mathcal{D}'(\Omega)$ to a test function $\varphi\in \mathcal{D}(\Omega)$ we write $(u,\varphi)$. The same notation will be used if $u\in \mathcal{S}'(\real^d)$ and $\varphi\in \mathcal{S}(\real^d)$ (and also for the inner product in $L_2(\Omega)$).  For $u\in \mathcal{D}'(\Omega)$  and a multi-index $\alpha = (\alpha_1, \ldots,\alpha_d)\in \nat_0^d$, we write $D^{\alpha}u$ for the $\alpha$-th {\em generalized} or {\em distributional derivative} of $u$ with respect to $x=(x_1,\ldots, x_d)\in \Omega$, i.e., $D^{\alpha}u$ is a distribution on $\Omega$, uniquely determined by the formula   
\[
(D^{\alpha}u,\varphi):=(-1)^{|\alpha|}(u,D^{(\alpha)}\varphi), \qquad \varphi \in \mathcal{D}(\Omega). 
\]
{In particular, if  $u\in L^1_{\text{loc}}(\Omega)$ and  there exists a function $v\in L^1_{\text{loc}}(\Omega)$ such that 
\[
\int_\Omega v(x)\varphi(x)\ud x=(-1)^{|\alpha|}\int_{\Omega}u(x)D^{(\alpha)}\varphi(x)\ud x \qquad \text{for all} \qquad \varphi \in \mathcal{D}(\Omega), 
\]
we say that $v$ is the {\em $\alpha$-th weak derivative} of $u$ and  write $D^{\alpha}u=v$. 
}
We also use the notation $
\frac{\partial^k}{\partial x_j^k}u:=D^{\beta}u
$ as well as $\textcolor{black}{D^{k}_ju:=}\partial_{x_j^k}u:=D^{\beta}u$,   for some 
multi-index  $\beta=(0,\ldots, k, \ldots,0)$ with $\beta_j=k$, $k\in \nat$. 
%Furthermore, for $m\in \nat_0$, we write $D^mu$ for any (generalized) $m$-th order derivative of $u$, where $D^0u:=u$ and $Du:=D^1u$. Sometimes we shall use subscripts such as $D^m_x$ or  $D^m_t $ to emphasize that we only take derivatives with respect to $x=(x_1, \ldots, x_d)\in \Omega$ or $t\in \real$. 

\section{Anisotropic function spaces}

Compared to classical (isotropic) function spaces, the smoothness properties of an element in an anisotropic function space depend on a chosen direction in $\rd$. In order to capture this phenomenon, let 
 us fix throughout the paper an {\em anisotropy} $\bfa=(a_1,\ldots, a_d)\in \rd_+$   normalized by 
\beq\label{anisotropy}
\left(\frac{1}{a_1}+\ldots + \frac{1}{a_d}\right)=d. 
\eeq
Moreover, we denote by 
\begin{equation}\label{aniso-dist}
|x|_{\bfa}:= \sum_{j=1}^d |x_j|^{a_j}, \qquad x=(x_1,\ldots, x_d)\in \real^d, 
\end{equation}
the {\em anisotropic pseudo-distance} corresponding to $\bfa$. 

\subsection{Anisotropic Sobolev spaces}

Let $\mathcal{O}\subset \rd$ be a domain, $1<p<\infty$, and $\bfl =(l_1,\ldots, l_d)\in \nat_0^{d}$. Then 
\beq\label{aniso-Sob}
W^{\bfl}_p(\mathcal{O})=\left\{f\in L_p(\mathcal{O}):\ \|f|W^{\bfl}_p(\mathcal{O})\|:= \|f|L_p(\mathcal{O})\|+\sum_{i=1}^d \left\|\frac{\partial^{l_i}f}{\partial x_i^{l_i}}\Big|L_p(\mathcal{O})\right\|<\infty\right\}
\eeq
is an anisotropic Sobolev space. If $l_1=\ldots =l_d=l$, then $W^{\bfl}_p(\mathcal{O})=W^l_p(\mathcal{O})$ is the usual (isotropic Sobolev space). We see that in contrast to the usual Sobolev spaces, the smoothness properties of an element of an anisotropic Sobolev space depend in general on the chosen direction in $\rd$. For $\bfalpha=\alpha \bfa$ with  $\alpha\in \real$ and $\bfa$ as in \eqref{anisotropy}, corresponding anisotropic Bessel potential spaces $H^{\bfalpha}(\rd)$ (=$H^{\alpha\bfa}(\rd)$) can be defined   via 
\[
H^{\alpha \bfa}(\rd):=\left\{f\in L_2(\rd):  \  \|f|H^{\alpha \bfa}(\rd)\|:=\left\|(1+|\xi|_{\bfa}^2)^{\alpha/2}\hat{f}(\xi)| L_2(\rd)\right\|<\infty \right\}. 
\]

\begin{remark}{
For   the regularity studies in \cite{AG12}  the authors were mainly interested in solutions of the  homogeneous heat equation $\partial_t u-\Delta u=0$. Therefore,  special attention was paid to the anisotropic Sobolev spaces  $W^{2,1}_p(\Omega)$ normed by 
\[
\left\|u|W^{2,1}_p(\Omega)\right\|
:=\left\|u|L_p(\Omega)\right\|+\sum_{i=1}^d \left\|\frac{\partial}{\partial{x_i}}u\Big|L_p(\Omega)\right\|
+ \sum_{i,j=1}^d \left\|\frac{\partial^2}{\partial{x_i\partial x_j}}u\Big|L_p(\Omega)\right\|
+ \left\|\frac{\partial}{\partial{t}}u\Big|L_p(\Omega)\right\|, 
\]
defined on the space-time cylinder  $\Omega=D\times[0,T]$, where   $D\subset \rd$ is some Lipschitz domain. In particular, these spaces coincide with our anisotropic Sobolev spaces $W^{\bfl}_p(\mathcal{O})$ if we replace $\mathcal{O}\subset\rd$ by $\Omega\subset \real^{d+1}$ in \eqref{aniso-Sob} and put $\bfl=(2,\ldots, 2,1)\in \nat_0^{d+1}$. 
%\todo{the fact that the norms are equivalent (i.e., only the highest derivatives have to be considered) should follow as in the isotropic case using the definition of $H^{\bfl}(\rd)$; should we add a remark?}
}
\end{remark}

\subsection{Anisotropic Besov spaces, wavelet decompositions}

\textcolor{black}{
We first recall the definition of anisotropic Besov spaces on $\Omega\subset \rd$. Whenever $f$ is a function in $\Omega$, we denote by $\Delta^k_hf$ the difference of order $k\geq 1$ and step $h\in \rd$, defined iteratively via 
\[
(\Delta^{\Omega}_hf)(x)=
\left.\begin{cases}
f(x+h)-f(x), & \text{if } x, x+h\in \Omega \\
0, & \text{otherwise} 
\end{cases}\right\}\quad \text{and} \quad (\Delta^{k,\Omega}_h f)(x)=\Delta^{\Omega}_h(\Delta^{k-1,\Omega}_h f)(x), \quad x\in \Omega. 
\]
%\[
%\Delta^{\Omega}_hf(x)=f(x+h)-f(x)\quad \text{and} \quad (\Delta^{k,\Omega}_h f)(x)=\Delta_h(\Delta^{k-1}_h)(x), \quad x\in \rd. 
%\]
If  ${\bf k}=(k_1,\ldots, k_d)$ is a multi-index with $k_i\geq 0$, we define the {\em iterated difference of order} ${\bf k}$ by 
\[
\Delta^{{\bf k}, \Omega}_h f(x)=\left(\Delta^{k_1, \Omega}_{h_1e_1}\circ \dots \circ \Delta^{k_d,\Omega}_{h_de_d} f\right)(x), 
\]
where $e_1,\ldots, e_d$ denotes the canonical basis of $\rd$. Moreover, let  $\bfalpha=\alpha \bfa=(\alpha_1,\ldots, \alpha_d)$ with  $\alpha>0$, $\bfa$ as in \eqref{anisotropy} and let  $0< p,q<\infty$. We say that $f\in L_p(\Omega)$ belongs to the anisotropic Besov space ${\bf B}^{\bfalpha}_{p,q}(\Omega)$ if the semi-norm
\[
|f|_{{\bf B}^{\bfalpha}_{p,q}(\Omega)}=\sum_{i=1}^d \left(\int_0^{\infty}t^{\alpha_i}\|\Delta^{k_i,\Omega}_{te_i}f|L_p(\Omega)\|^q \frac{dt}{t}\right)^{1/q}
\]
is finite (here $k_i$ are integers such that $k_i>\alpha_i$, $i=1,\ldots, d$). Moreover, the norms 
\[
\|f|{\bf B}^{\bfalpha}_{p,q}(\Omega)\|:=\|f|L_p(\Omega)\|+|f|_{{\bf B}^{\bfalpha}_{p,q}(\Omega)}, 
\]
are known to be equivalent for any choice $k_i>\alpha_i$. Finally, the isotropic Besov spaces $B^s_{p,q}(\Omega)$ are nothing but ${\bf B}^{\bfalpha}_{p,q}(\Omega)$ if $\bfalpha=(s,\ldots, s)$. 
For our studies below it will be convenient to use another approach and define} anisotropic Besov spaces $B^{\alpha \bfa}_{p,q}(\Omega)$ via  wavelet decompositions,  valid for the whole range $0<p,q<\infty$. 
In particular,  our wavelet approach is based on compactly supported wavelets and a dilation adapted to the anisotropy of the spaces. 
Such a characterization of anisotropic Besov spaces was developed in \cite{GHT04} with the  forerunners \cite{GT02,Hoch02}. Note that we adapt the results presented there according to our needs. 
%Due to the different contexts (isotropic) Besov spaces arose from they can be defined/characterized in several ways, e.g. via higher order differences, the Fourier-analytic approach or  decompositions with suitable building blocks, cf. \cite{Tri83, Tri08} and the references therein. Under certain restrictions on the parameters these different approaches might even coincide. 
%In particular, the characterization of Besov norms as weighted sums of wavelet coefficients has important applications in data compression, nonlinear approximation and the numerical resolution of elliptic partial differential equations. Extensions wavelet decompositions for anisotropic Besov spaces were studied in \cite{GHT04} based on the forerunners \cite{GT02,Hoch02}. 
The wavelet system we are looking for will be dilated by a matrix $M$, where 
\begin{equation}\label{M-compatible}
M:=\mathrm{diag}\left(\lambda^{1/a_1},\ldots, \lambda^{1/a_d}\right) \quad \text{for some }\quad  \lambda>1, 
\end{equation}
which is 'compatible' with the anisotropy $\bfa$ in the sense that one recovers the correct homogeneity over Besov semi-norms, i.e., 
\[
|\det M|^{1/p}|f(M\cdot)|_{{\bf B}^{\alpha\bfa}_{p,q}}=\lambda^{\alpha}|f|_{{\bf B}^{\alpha \bfa}_{p,q}}. 
\]
In particular, \textcolor{black}{also with this approach} we recover the isotropic Besov spaces $B^{s}_{p,q}(\Omega)$ based on dyadic dilations by setting  $a_1=\ldots=a_d=1$, $\alpha=s$, and $\lambda=2$. \\

\textcolor{black}{We briefly recall our  wavelet approach  based on multi-resolution analysis: } 
 For our definition  of the anisotropic Besov spaces, we will use  compactly supported wavelets constituting Riesz-bases in $L_2(\mathbb{R})$,
 that are obtained by dilating, translating and scaling a  fixed function, the so--called  {\em mother wavelet} $\psi$. 
 This mother wavelet is usually constructed by means of a  {\em multiresolution analysis (MRA)}  that is, a  sequence   $\{V_j\}_{j \in \mathbb{Z}}$   of shift-invariant, closed subspaces of $L_2(\mathbb{R})$ whose union is dense in $L_2$ while their intersection is zero. Moreover, all the spaces are related via   dilation, and the space   $V_0$ is spanned  by the translates of  a fixed function $\phi$,  called the {\em generator} or {\em  father wavelet}. 
 \textcolor{black}{We put  $\psi^0:=\phi$ and $\psi^1:=\psi$   and denote by $U$ the nontrivial vertices of the square $[0,1]^d$. Then 
% In particular,  
by taking tensor products, i.e.,  
\[
\psi^u(x_1,\ldots, x_d):=\prod_{j=1}^d \psi^{u_j}(x_j), \quad u=(u_1,\ldots, u_d)\in U, 
\]
%\[
%\phi=\phi_1(x_1)\dots \phi_d(x_d)\quad \text{and}\quad \psi=\psi_1(x_1)\dots %\psi_d(x_d)
%\]
}
a compactly supported  basis for $L_2(\mathbb{R}^d)$ can be constructed. In contrast to the isotropic case our wavelets are constructed such that they are well  adapted to the anisotropy $\bfa$, which is achieved by using the diagonal dilation Matrix $M$ from \eqref{M-compatible} compatible with $\bfa$. For this reason we will call them  $M$-wavelets in the sequel. 

The existence of compactly supported scaling functions  (and wavelets) for an arbitrary dilation matrix $M$ is a delicate matter. Concrete examples when $M$ has a relatively simple form can be found in \cite{Ay99,HL99,HRZ99}.   
\textcolor{black}{However, since we consider tensor products of wavelets the situation simplifies considerably in our context. In this case $M$ is diagonal and we only dilate differently in different directions. Additionally, we may} assume that $M$ is integer valued and put $m=|\det M|=\lambda^d$. 
Note that from the  discussion in \cite[Sect. 3.3]{GT02} it follows that this is not a severe restriction in our construction since  for all anisotropies $\bfa\in \mathbb{Q}^d_+$ there exists a number $\lambda>1$ such that $\lambda^{1/a_1}, \ldots, \lambda^{1/a_d}\in \nat$.   

We now explain what  we call an {\em admissible biorthogonal $M$-wavelet bases} in the sequel. For the precise construction we refer to \cite{GT02, GHT04}. 
%For the results we wish to present in this paper, any couple constructed as above will suffice. However, our results hold for more general scaling functions that fit into the definition of admissibility. 
%Let $\varphi\in H^s{\rd}$ be a compactly supported scaling function of sufficiently high regularity $s$ and let $\$
Let ${\phi}$ be a compactly supported scaling function, the {\em father wavelet}, of tensor product type on $\real^d$  having sufficiently high smoothness  and let $\Psi'=\{\psi_i: \ i=1,\ldots, m-1\}$ be the set containing the corresponding multivariate mother wavelets such that, for a given $L\in \nat$ with $L>d/2$ and some $N>0$ the following requirements hold:  For all $\psi\in \Psi'$, 
\begin{align}
\supp{\phi}, \ \supp \psi   & \ \subset \ [-N,N]^d, \label{wavelet-1}\\
{\phi}, \ \psi  & \ \in \ H^{L\bfa}(\real^d), \label{wavelet-2}\\
%\int_{\real^d} x^{\alpha}\psi(x)\ud x & \ =\ 0 \quad \text{ for all } \alpha \in \nat_0^d \ \text{ with } \  \ |\alpha|\leq r. 
\psi & \ \perp\  \Pi_{L-1}:=\mathrm{span}\left\{x^{\bfl} =x_1^{l_1}\cdots x_d^{l_d}: \ |\bfl|=l_1+\ldots+l_d\leq L-1 \right\}. 
\label{wavelet-3}
\end{align}  
%\todo{check moment conditions! Maybe it is better to use wavelet system with the properties from \cite[Thm. 1.2]{GT02}, i.e., $\psi\in B^{\alpha_0\bfa}_{p,q}(\rd)\cap H^{L\bfa}(\rd)$ !}
In particular,  \eqref{wavelet-3} guarantees that the mother wavelets $\psi$ are orthogonal to the polynomials $\Pi_{L-1}$ of order less than $L$, which is possible by the  assumptions \eqref{wavelet-1} and \eqref{wavelet-2},  %, i.e.,  
%\[
%\Pi_{r-1}=\mathrm{span}\left\{x^{\bfr} =x_1^{r_1}\cdots x_d^{r_d}: \ %|\bfr|=r_1+\ldots+r_d\leq r-1 \right\} 
%\]
%{Alternatively: Maybe it is better to use 
%\[
%\mathbb{P}^{\bfr}=\mathrm{span}\left\{x^{\bfr} =x_1^{s_1}\cdots x_d^{s_d}: \ s_i<r_i, \ %i=1,\ldots, d \right\} 
%\]
%}
cf.  \cite[Prop.~3.3]{GT02}. 
Moreover, by $\mathcal{D}^{+}$ we denote the set of all  \textcolor{black}{cuboids} in $\real^d$ with measure at most $1$ of the form 
\[
\mathcal{D}^{+}:=\left\{I\subset \real^d: \ I=M^{-j}([0,1]^d+k), \ j\in \nat_0, \ k\in \mathbb{Z}^d\right\}
\]
and we set $\mathcal{D}_j:=\{I\in \mathcal{D}^+: \ |I|=\lambda^{-jd}\}.$ 
For the  shifts and dilations of the father wavelet and the corresponding wavelets we use the abbreviations 
\begin{equation}\label{wavelet-4}
{\phi}_k(x):={\phi}(x-k), \quad \psi_{I}(x):=|\det M|^{j/2}\psi(M^jx-k) \qquad \text{for}\quad  j\in \nat_0, \ k\in \mathbb{Z}^d, \ \psi\in \Psi'. 
\end{equation}
It follows that 
\[
\left\{{\phi}_k, \ \psi_{I}:  \ k\in \mathbb{Z}^d, \  I\in \mathcal{D}^+, \  \psi\in \Psi'\right\}
\]
is a Riesz basis in $L_2(\real^d)$. 
\textcolor{black}{Furthermore, we assume that there exists a dual basis also constructed by means of an MRA $\{\tilde{V}_j\}_{j\in \mathbb{Z}}$, i.e.,  functions $\tilde{\phi}$ and $\tilde{\psi}\in \tilde{\Psi}^{'}=\{\tilde{\psi}_i: \ i=1,\ldots, m-1\}$ satisfying 
\begin{align}
    \langle \tilde{\phi}_k, \psi_I\rangle& =\langle  \tilde{\psi}_I, {\phi}_k\rangle=0, \\
     \langle \tilde{\phi}_k, \phi_l\rangle& =\delta_{k,l} \qquad (\text{Kronecker symbol}), \\
     \langle \tilde{\psi}_I, \psi_{I'}\rangle& =\delta_{I,I'}.  
\end{align}
The dual Riesz basis should fulfil the same requirements as the primal Riesz basis, i.e., 
\begin{align}
     \supp \tilde{\phi}, \ \supp \tilde{\psi} &\subset [-N,N]^d, \\
     \tilde{\phi}, \ \tilde{\psi}& \in H^{L\bfa}(\rd), \\
     \tilde{\psi}& \perp \Pi_{L-1}. 
\end{align}
}
%Further, the dual Riesz basis should fulfill the same requirements, i.e., there exist functions $\tilde{\phi}$ and $\tilde{\psi}\in \tilde{\Psi}^{'}=\{\tilde{\psi}_i: \ i=1,\ldots, m-1\}$ such that 
%\begin{align}
%    \langle \tilde{\phi}_k, \psi_I\rangle& =\langle  \tilde{\psi}_I, {\phi}_k\rangle=0, \\
%     \langle \tilde{\phi}_k, \phi_l\rangle& =\delta_{k,l} \qquad (\text{Kronecker symbol}), \\
%     \langle \tilde{\psi}_I, \psi_{I'}\rangle& =\delta_{I,I'}, \\
%     \supp \tilde{\phi}, \ \supp \tilde{\psi} &\subset [-N,N]^d, \\
%     \tilde{\phi}, \ \tilde{\psi}& \in H^{L\bfa}(\rd), \\
%     \tilde{\psi}& \perp \Pi_{L-1}. 
%\end{align}

Denote by $Q(I)$ some  \textcolor{black}{cuboid} (of minimal size) such that $\supp \psi_I \subset Q(I)$ for every $\psi\in \Psi'$. Then, we may assume that $Q(I)=M^{-j}k+M^{-j}Q$ for some  \textcolor{black}{cuboid} $Q$. Put $\Lambda'=\mathcal{D}^{+}\times \Psi'$.  
Then, every function $f\in L_2(\real^d)$ can be written as 
\[
f=
%%P_0f+\sum_{(I,\psi)\in \Lambda'}\langle f, {\psi}_I\rangle \psi_I:=
\sum_{k\in \mathbb{Z}^d}\langle f,{\tilde{\phi}}_k\rangle {{\phi}}_k +\sum_{(I,\psi)\in \Lambda'}\langle f, {\tilde{\psi}}_I\rangle \psi_I.  
\]
It will be convenient to include ${\phi}$ into the set $\Psi'$. We use the notation ${\phi}_I:=0$ for $|I|<1$, ${\phi}_I={\phi}(\cdot-k)$ for $I=k+[0,1]^d$, and can simply write 
\begin{equation}\label{dec-help}
f=\sum_{(I,\psi)\in \Lambda}\langle f, \tilde{\psi}_I\rangle \psi_I, \qquad \Lambda=\mathcal{D}^+\times \Psi, \quad \Psi=\Psi'\cup \{{\phi}\}.
\end{equation}

The two systems $\{\phi_k, \psi_I\}_{k,I}$ and $\{\tilde{\phi}_k,\tilde{\psi}_I\}_{k,I}$ constructed as above are said to be a {\em pair of admissible biorthogonal $M$-wavelet bases} and they may be used to obtain decompositions of many classical function spaces. In particular, 
according to \cite[Thm.~1.2]{GHT04} and \cite[Thm.~1.2]{GT02} anisotropic Besov spaces on $\real^d$ can be characterized by decay properties of the wavelet coefficients, if the parameters fulfill certain conditions. This characterization motivates the following definition.  \\

\begin{definition}[{Anisotropic Besov spaces, wavelet decompositions}]\label{def-B-aniso} \phantom{m}\\
Let $\bfalpha=(\alpha_1,\ldots, \alpha_d)\in \rd_+$ and      $0<p,q<\infty$.  Moreover, let $\bfalpha=\alpha \bfa$, where 
$\alpha>\max\left\{0,d(1/p-1)\right\}$ and the anisotropy $\bfa$ is  normalized as in \eqref{anisotropy}. Let $M$ be a dilation matrix compatible with $\bfa$. We assume that  $\{\phi_k, \psi_I\}_{k,I}$ and $\{\tilde{\phi}_k,\tilde{\psi}_I\}_{k,I}$ is a pair of biorthogonal admissible $M$-wavelet bases with  $\phi, \tilde{\phi}, \psi, \tilde{\psi}\in H^{L\bfa}(\rn)$ for some integer $L>\max\{d/2, \alpha_1,\ldots, \alpha_n\}$. %Furthermore, assume that . Choose $r\in \nat$ such that $r>\alpha \cdot  \max\{ a_1,\ldots, a_d\}$  and construct a wavelet Riesz basis as described above. 
Then \textcolor{black}{the Besov space $B^{\alpha \bfa}_{p,q}(\real^d)$ ($=B^{\bfalpha}_{p,q}(\real^d)$) is defined as the set of all functions $f\in L_p(\real^d)$ satisfying }
%a function $f\in L_p(\real^d)$ belongs to the Besov space $B^{\alpha \bfa}_{p,q}(\real^d)$ ($=B^{\bfalpha}_{p,q}(\real^d)$)  if, and only if, 
\begin{equation}\label{besov-decomp}
f=\sum_{k\in \mathbb{Z}^d}\langle f,{\tilde{\phi}}_k\rangle {\phi}_k +\sum_{(I,\psi)\in \Lambda'}\langle f, \tilde{\psi}_I\rangle \psi_I  
\end{equation}
(convergence in $\mathcal{S}'(\real^d)$) with 
\begin{align}
\|f|B^{\alpha \bfa}_{p,q}(\real^d)\| 
&\sim  \left(\sum_{k\in \mathbb{Z}^d} |\langle f,{\tilde{\phi}}_k\rangle|^p\right)^{1/p} + \notag\\
& \qquad   \left(\sum_{j=0}^{\infty}|\det M|^{j\left(\frac{\alpha}{d}+(\frac 12-\frac 1p)\right)q}\left(\sum_{(I,\psi)\in \mathcal{D}_j\times \Psi'}|\langle f, \tilde{\psi}_{I}\rangle|^p\right)^{q/p}\right)^{1/q}<\infty.\label{besov-norm}
\end{align}
\end{definition}

{
\begin{remark}{\label{rem-bf=b}
\bit 
\item[(i)]
 In particular, for the adaptivity scale  $B^{\alpha \bfa}_{\tau,\tau}(\real^d)$ with $\alpha=d\left(\frac{1}{\tau}-\frac 1p\right)$,  we see that the quasi-norm  \eqref{besov-norm} becomes 
\begin{align}
\|f|B^{\alpha \bfa }_{\tau,\tau}(\real^d)\|&\sim  \left(\sum_{k\in \mathbb{Z}^d} |\langle f,{{\phi}}_k\rangle|^{\tau}\right)^{1/\tau} +   
   \left(\sum_{j=0}^{\infty}|\det M|^{j\left(\frac 12-\frac 1p\right)\tau}\sum_{(I,\psi)\in \mathcal{D}_j\times \Psi'}|\langle f, {\psi}_{I}\rangle|^{\tau}\right)^{1/\tau}. \notag\label{besov-norm2}
\end{align}
%Also note here that that weights $|\det M|^{j\left(\frac 12-\frac 1p\right)}$ do not appear in \cite[Thm.~1.2]{GHT04} due to the fact that we scaled our wavelets in a different manner in \eqref{wavelet-4}. % wavelets form an orthonormal basis in $L_2(\rn)$. 
\item[(ii)] \textcolor{black}{From  \cite[Thm. 1.2]{GT02}  we deduce that ${\bf B}^{\alpha \bfa}_{p,q}(\rd)= B^{\alpha \bfa}_{p,q}(\rd)$ for the range of parameters 
$$\alpha>0,\quad 1\leq p,q<\infty, $$ 
whereas  \cite[Thm. 1.2]{GHT04} additionally covers the case 
$${\bf B}^{\alpha \bfa}_{\tau,\tau}(\rd)=B^{\alpha \bfa}_{\tau,\tau}(\rd),\qquad  
\alpha>\max\left\{0, d\left(\frac{1}{\tau}-1,0\right)\right\},\quad \frac{1}{\tau}=\frac{\alpha}{d}+\frac 1p, \quad 0<\tau<\infty. 
$$}
Thus, we see that the range of spaces we consider in Definition \ref{def-B-aniso} is larger.  The restriction $\alpha>\max\left\{0,d(1/p-1)\right\}$ is necessary since it guarantees that our anisotropic Besov spaces considered in Definition \ref{def-B-aniso}  satisfy $B^{\alpha \bfa}_{p,q}(\rd)\hookrightarrow L_{\max\{1+\varepsilon,p\}}(\rd)$, see also   \cite[Cor. 5.4]{GHT04} in this context. 
\item[(iii)] Interpretation: From the above construction of the anisotropic spaces we see that  $\alpha$ describes mean smoothness and $\bfa$ measures the anisotropy. 
\item[(iv)] If $p=q=2$, then $B^{\alpha \bfa}_{2,2}(\rd)$ coincides with the anisotropic Bessel potential space, i.e., we have 
$$\textcolor{black}{({\bf B}^{\alpha \bfa}_{2,2}(\rd)= )} \ B^{\alpha \bfa}_{2,2}(\rd)=H^{\alpha \bfa}(\rd). $$ 
Furthermore, if $\alpha \bfa=\bfl =(l_1,\ldots, l_d)$ is an integer-valued multi-index, one recovers the  anisotropic Sobolev spaces 
\[
{({\bf B}^{\bfl}_{2,2}(\rd)=)} \  B^{\bfl}_{2,2}(\rd)=H^{\bfl}(\rd)=W^{\bfl}_2(\rd). 
\]
\item[(v)] As already mentioned before, we adapted the results from \cite[Thms. 1.2]{GT02, GHT04} slightly. %To be precise, this means that we 'avoided'  the problem of compactly supported wavelets in  higher dimensions by considering tensor products of 1-dimensional wavelets.  Moreover, our wavelets are not minimal smooth which allows us to  drop the  condition $r>d/2$, since our wavelets are assumed to be continuous. %Note that the restriction $L>\{\alpha a_1,\ldots, \alpha a_n\}$ in \cite{GT02} is not correct if one assumes the wavelets $\psi$ to belong to $H^{L\bfa}(\rn)$. This follows from the proof of \cite[(4.3)]{GT02} on page 14, where for smoothness $L\bfa=$ the correct computation reads as 
%$$
%\frac{l_1 }{\alpha_1}+\ldots + \frac{l_d }{\alpha_d}=\frac{l_1}{\alpha}+\ldots \frac{l_d}{\alpha}=\frac{L}{\alpha}>1. 
%$$  
%Putting $r=L$ yields the condition  $r>\alpha$. 
In particular, we consider wavelets which form a biorthogonal basis for $L_2(\rd)$ instead of $L_p(\rd)$ which leads to the weight  factor  $|\det M|^{j\left(\frac{\alpha}{d}+(\frac 12-\frac 1p)\right)}$ regarding the decay of the wavelet coefficients $\langle f, \tilde{\psi}_{I}\rangle$ in \eqref{besov-norm} instead of $|\det M|^{\frac{j\alpha}{d}}$ in \cite[Thm.~1.2]{GT02}.  \\
Moreover, the additional condition that $\phi,\tilde{\phi}\in H^{L\bfa}(\rd)\cap B^{\alpha_0\bfa}_{p,q}(\rd)$ for some $\alpha_0>0$ can be circumvented by choosing $L$ large enough. This can be seen as follows: Since $\phi,\tilde{\phi}$ have compact support we deduce from the the definition of the spaces  that for  $p<2$ and $0<q\leq \infty$, 
\[
\phi, \tilde{\phi}\in H^{L\bfa}(\rd)=B^{L
\bfa}_{2,2}(\rd)\hookrightarrow B^{\alpha_0\bfa}_{p,q}(\rd). 
\]
For the case that  $2\leq p$ we use \cite[Thm. 18.4]{BNI80a} and obtain
\[
H^{L\bfa}(\rd)=B^{L
\bfa}_{2,2}(\rd)\hookrightarrow B^{\alpha_0\bfa}_{p,q}(\rd)\quad \text{if}\quad L>\alpha_0+d\left(\frac 12-\frac 1p\right). 
\]
\item[(vi)] In particular, if  $1\leq p,q< \infty$ our anisotropic Besov spaces $B^{\alpha \bfa}_{p,q}$  coincide with the spaces from \cite{GT02, GHT04}. Therefore, by  \cite[Cor.~5.3]{GHT04} we have the following interpolation result: 
\begin{equation}\label{aniso-interpol}
\big(L_p(\rd), B^{\alpha\bfa}_{p,r}(\rd)\big)_{\theta,q}=B^{\theta\bfa}_{p,q}(\rd), \qquad 0<\theta<1, \quad 0<r<\infty. 
\end{equation}
\eit 
}
\end{remark}
}

Corresponding function spaces on domains $\mathcal{O}\subset \real^d$ can be introduced via restriction, i.e., 
\begin{eqnarray*}
B^{\alpha \bfa}_{p,q}(\mathcal{O})&=& \left\{f\in \mathcal{D}'(\mathcal{O}): \ \exists g\in B^{\alpha \bfa}_{p,q}(\real^d), \ g\big|_{\mathcal{O}}=f \right\},\\
\|f|B^{\alpha \bfa}_{p,q}(\mathcal{O})\|&=& \inf_{g|_{\mathcal{O}}=f}\|f|B^{\alpha \bfa}_{p,q}(\real^d)\|. 
\end{eqnarray*}

\subsection{Domains allowing  extensions}

 In what follows we want to investigate anisotropic function spaces on more general domains $\Omega\subset \rd$. So far we introduced anisotropic spaces on $\rd$, where a lot of the tools we need (in particular, wavelet decompositions of anisotropic Besov spaces) are available. Then corresponding spaces on domains can be defined via restriction.  
 Now, in order to truly establish our results on domains $\Omega$,  we need an extension operator for our anisotropic  spaces. Such extensions of anisotropic spaces defined on   $\Omega$ to the whole $\rd$ are possible if $\Omega$ satisfies what is called a strong $\bfr$-horn condition. 
 In order to explain this condition we need some notation. Let $\bfr=(r_1,\ldots, r_d)$ be a vector with positive components. Suppose that $0<h\leq \infty$, $\varepsilon >0$, and $\bfb\in \rd$ with $b_i\neq 0$ for $i=1,\ldots, d$. The set 
 \[
 V(\bfr,h,\varepsilon,\bfb):=\bigcup_{0<\nu<h}\left\{
 x: \ \frac{x_i}{b_i}>0,  \ \nu<\left(\frac{x_i}{b_i} \right)^{r_i}<(1+\varepsilon)\nu\quad \text{for}\quad i=1,\ldots, d
 \right\}
 \]
 is called an {\em $\bfr$-horn of radius $h$ and opening $\varepsilon$}. \\

 \begin{minipage}{0.43\textwidth}
 The diagram aside illustrates $\bfr$-horns for different parameters $\bfr$ and $\bfb$ in $\real^2$: 
 \begin{eqnarray*}
 V_1&=& V\big((1,1),h,\varepsilon, (1,1)\big);\\[0.2cm] 
 V_2&=& V\big((1,1),h,\varepsilon, (1,2)\big);\\[0.2cm] 
 V_3&=& V\big((2,1),h,\varepsilon, (1,1)\big).
 \end{eqnarray*}
 
 In particular, we see that in the isotropic case (see $V_1$ and $V_2$ with $r_1=r_2=1$) the horn is just a cone. If we have an anisotropy (see $V_3$ with $r_1=2$ and $r_2=1$) the different scaling exponent in the different directions causes the cone to become a horn. Moreover, the vector $\bfb$ specifies the exact location of the $\bfr$-horn in the coordinate system (compare $V_1$ with $\bfb=(1,1)$ with $V_2$ where $\bfb=(1,2)$). 
 
 \end{minipage}\hfill\begin{minipage}{0.55\textwidth}
 \includegraphics[width=9cm]{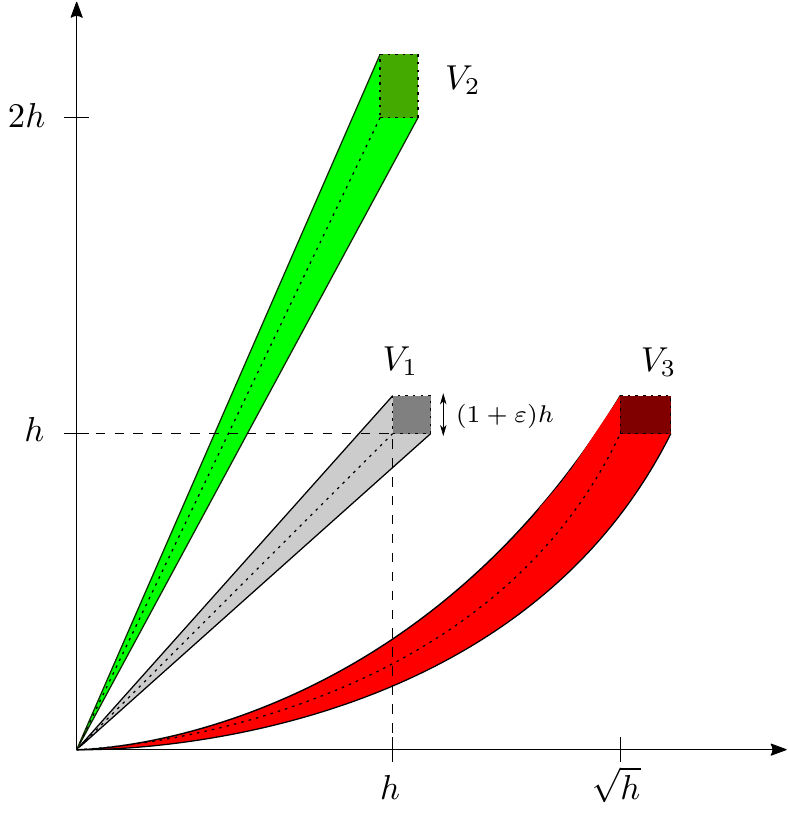}
 \end{minipage}
 
 An open set $\Omega\subset \rd$ is said to satisfy a {\em weak $\bfr$-horn condition} if there is a positive integer $N$ such that for each $j\in \{1,\ldots, N\}$, there are open sets $\Omega_j$ and an $\bfr$-horn $V_j(\bfr,h,\varepsilon, \bfb^{(j)})$ such that 
 \begin{equation}\label{weak-horn}
 \Omega=\bigcup_{j=1}^N\Omega_j=\bigcup_{j=1}^N\left(\Omega_j+ V_j(\bfr,h,\varepsilon, \bfb^{(j)})\right). 
 \end{equation}
 The relation \eqref{weak-horn}  expresses the fact that for any point $x\in \Omega_j$, if the horn $V_j(\bfr,h,\varepsilon, \bfb^{(j)})$ is shifted parallel to itself in such a way that its vertex coincides with $x$, then the resulting shifted horn lies in $\Omega$.  
 If, in addition, there exists $\delta>0$ such that 
 $$\Omega=\bigcup_{j=1}^N \Omega_j^{(\delta)}, \quad \text{where} \quad \Omega_j^{(\delta)}=\left\{x\in \Omega_j: \ \mathrm{dist}(x,\Omega\setminus \Omega_j)>\delta\right\},$$ 
 then $\Omega$ is said to satisfy a {\em strong $\bfr$-horn } condition. Note that if $r_1=r_2=\ldots=r_d$, every $\bfr$-horn is a cone. It is possible in this case to show that the concept of a domain having a Lipschitz boundary coincides with the concept of a domain satisfying the $\bfr$-horn condition.

 The following theorem can be found in \cite[Thm.~9.6]{BNI80} and \cite[Thm. 2, p.~382]{Nik75}.

 \begin{theorem}\label{thm:ext-op}
 Suppose  $\Omega\subset \rd$ satisfies a strong $\bfr$-horn condition and  $1\leq p,q\leq \infty$. %\todo{What about $p,q<1$?}
 \begin{itemize}
     \item[(i)] Let $1<p<\infty$. Then $W^{\bfr}_p(\Omega)$ is the set of all functions which are the restrictions to $\Omega$ of elements of $W^{\bfr}_p(\rd)$. In particular, there is a bounded, linear extension map $E: W^{\bfr}_p(\Omega)\rightarrow W^{\bfr}_p(\rd)$. 
     \item[(ii)]  Furthermore, ${\bf B}^{\bfr}_{p,q}(\Omega)$ is the set of all functions which are the restrictions to $\Omega$ of elements of 
     ${\bf B}^{\bfr}_{p,q}(\rd)$. In particular, there is a bounded, linear extension map $E: {\bf B}^{\bfr}_{p,q}(\Omega)\rightarrow {\bf B}^{\bfr}_{p,q}(\rd)$. 
  \end{itemize}
 \end{theorem}

 {
 In view of Remark \ref{rem-bf=b}(ii) we immediately obtain the following result. 
  \begin{corollary}\label{cor:ext-op}
 Suppose  $\Omega\subset \rd$ satisfies a strong $\bfr$-horn condition and  $1\leq p,q< \infty$. Then  there is a bounded, linear extension map $E: {B}^{\bfr}_{p,q}(\Omega)\rightarrow {B}^{\bfr}_{p,q}(\rd)$. 
 \end{corollary}
 }

\begin{remark}{We provide examples of domains satisfying the (weak or strong) $\bfr$-horn condition:  
%are given in \cite[p.~155]{Nik75}. 
%In particular, f
For $d=2$,  $\Omega=\real^2$ as well as the rectangular parallelepiped 
\[
\Omega=\left\{(x_1,x_2): \  |x_1|<a, \ |x_2|<b \right\},
\]
where $a,b>0$, satisfy the strong $\bfr$-horn condition for any $\bfr$. Moreover, the disk 
\[
Q=\left\{(x_1,x_2):\ x_1^2+x_2^2<1\right\}
\]
satisfies the weak $\bfr$-horn condition only if $\frac 12 r_1\leq r_2\leq 2r_1$ and the strong $\bfr$-horn condition only if $r_1=r_2$. \\
Since in the isotropic case, if $r_1=\ldots, r_d$ a  Lipschitz domain satisfies the strong $\bfr$-horn condition for any $\bfr$, from the product structure of  the space-time cylinder $\Omega=D\times [0,T]\subset \real^{d+1}$, where $D\subset \rd$ is a bounded Lipschitz domain, we deduce that $\Omega$  satisfies the strong $\bfr$-horn condition for any $\bfr$ of the form $\bfr=(r,\ldots, r, \frac r2)$, where $r>0$. 
}
\end{remark}

\begin{remark}{The extension operator from Theorem \ref{thm:ext-op} allows us to transfer many results (such as embeddings, interpolation, etc.), which are known for anisotropic spaces on $\rd$, to domains satisfying the horn condition. In particular, it allows us to relate the regularity spaces 
$$\mathbb{B}^{s}_{p}(\Omega):=\left(L_p(\Omega), W^{2,1}_p(\Omega)\right)_{\frac{s}{2},p}, \qquad 0<s<1, \quad 1<p<\infty,$$
appearing in \cite{AG12} (which for general $s>0$ can be defined via the action of the derivatives $\partial_t$ and $\partial_{x_j x_i}$)   
 to our spaces: According to \cite[Thm.~18.9]{BNI80a} for any $\Omega\subset \rd$ satisfying an $\bfl$-horn condition we have the embedding 
\begin{equation}\label{emb-W-B}
B^{\bfl}_{p,\min(p,2)}(\Omega)\hookrightarrow W^{\bfl}_p(\Omega)\hookrightarrow B^{\bfl}_{p,\max(p,2)}(\Omega). 
\end{equation}
Since the space-time cylinder $\Omega=D\times [0,T]$ satisfies the $\bfl$-horn condition for arbitrary $\bfl$ we deduce from \eqref{emb-W-B} and \eqref{aniso-interpol} that 
\begin{equation}\label{coinc-B-spaces}
\mathbb{B}^{s}_{p}(\Omega)=B^{s,\ldots, s,\frac s2}_{p,p}(\Omega)=B^{\tilde{s}\bfa}_{p,p}(\Omega), 
\end{equation}
using the anisotropy 
\begin{equation}\label{spec-aniso}
\bfa=\left(a_1,\ldots, a_{d+1}\right)=\left(\frac{d+2}{d}, \ldots, \frac{d+2}{d}, \frac 12\frac{d+2}{d} \right)=\frac{d+2}{d} \left(1,\ldots, 1,\frac 12\right)
\end{equation}
together with the mean smoothness $\tilde{s}=\frac{sd}{d+2}$. 
Moreover, choosing $s=p=q=2$ and  $\tilde{s}=\frac{2d}{d+2}$ yields the special case \begin{equation}\label{coinc-W-spaces}
   W^{2\ldots, 2,1}(\Omega)= B^{\frac{2d}{d+2}\bfa}_{2,2}(\Omega).
\end{equation}
}
\end{remark}

\subsection{Anisotropic Kondratiev spaces}

In this section, we introduce anisotropic Kondratiev spaces, which are special weighted anisotropic Sobolev spaces. The corresponding isotropic spaces play a central role in the regularity theory for elliptic PDEs on domains with piecewise smooth boundary, particularly polygons (2D) and polyhedra (3D). For a systematic treatment and further references  we refer to \cite{DHS17a}. In particular, in the isotropic case the weight is often chosen to be a power of the distance to the {\em singular set} of the boundary of a domain $\domain\subset \rd$, i.e., the set of all points $x\in \partial \domain$ for which for any $\varepsilon>0$ the set $\partial \domain\cap B_{\varepsilon}(x)$ is not smooth (here $B_{\varepsilon}(x)$ %:=\{y\in \rd: \ |x-y|<\varepsilon\}$ 
denotes the open ball in $\rd$ around a point $x$ with radius $\varepsilon>0$). \\
We adapt this idea and define now corresponding anisotropic Kondratiev spaces using weights which constitute powers of the anisotropic distance based on \eqref{aniso-dist} to a singular set $M\subset \partial \domain$.  

Precisely, let $\domain\subset \rd$ be a domain and let $M$ be a nontrivial closed subset of its boundary $\partial\domain$. Furthermore, let $1\leq p< \infty$, 
 $\bfm{=(m_1,\ldots, m_d)}=m\bfa\in \nat_{0}^d$ where the anisotropy $\bfa=(a_1,\ldots, a_d)$ is normalized as in \eqref{anisotropy}, and $\gamma\in \real$. Then the  anisotropic Kondratiev space $\mathcal{K}^{\bfm}_{p,\gamma}(\domain)$ (=$\mathcal{K}^{m\bfa}_{p,\gamma}(\domain)$)  is the collection of all $u\in \mathcal{D}'(\Omega)$ such that 

\[
\|u|\mathcal{K}^{m\bfa}_{p,\gamma}(\domain)\|
:=\left({\sum_{i=1}^d\sum_{\alpha_i\leq m_i}}\int_{\domain}\left|(\rho_{\bfa}(x))^{m-\gamma}D_i^{\alpha_i}{u}(x)\right|^p \ud x\right)^{1/p}<\infty, 
\]
where $\rho_{\bfa}(x)=\min(1, \mathrm{dist}_{\bfa}(x,M))$ and $\mathrm{dist}_{\bfa}$ denotes the anisotropic distance to  $M\subset \partial \Omega$, i.e., 
\[
\mathrm{dist}_{\bfa}(x,M)=\inf_{y\in M}|x-y|_{\bfa}\qquad \text{with} \qquad |x-y|_{\bfa}=\sum_{i=1}^d |x_i-y_i|^{a_i}. 
\]

%\remark{ 
%We remark that  the Besov (and Kondratiev) spaces we are working with are defined in the setting of distributions, i.e., as subsets of $\mathcal{D}'(\mathcal{O})$,  and therefore may contain 'functions'  which take complex values. However, when considering the heat equation or more general parabolic problems, we restrict ourselves to the  real-valued setting.  Therefore, the solutions are real-valued as well. 
%}

\begin{remark}{Later on we want to compare our results with the ones obtained in \cite{AG12} on the time space cylinder $\Omega=D\times[0,T]$. In this context we remark that the weight appearing in the gradient estimates in  \cite[Thm.~4]{AG12} is comparable to our weight $\rho_{\bfa}(x)$: It is (also) based on  powers of the so-called parabolic distance $\delta(x,t)$,  which is a special anisotropic distance to the parabolic boundary $$M:=\partial_{\mathrm{par}}\Omega:=(D\times \{0\})\cup (\partial D\times [0,T]).$$   To be precise, for $(x,t)\in \Omega=D\times[0,T]$ it is  defined as 
\[
\delta(x,t):=\inf\left\{\rho((x,t),(y,s)): \ (y,s)\in \partial_{\mathrm{par}}\Omega\right\}, \qquad \rho((x,t),(y,s))\sim |x-y|+\sqrt{|t-s|}. 
\]
Thus, for the special anisotropy \eqref{spec-aniso} we see that 
$$
|(x,t)-(y,x)|_{\bfa}= \sum_{i=1}^d |x_i-y_i|^{a_i}+|t-s|^{a_{d+1}} \sim \left(\sum_{i=1}^{d} |x_i-y_i| +\sqrt{|t-s|}\right)^{\frac{d+2}{d}},
$$
which yields 
\begin{equation}\label{comp-weights}
\rho_{\bfa}(x,t)\sim \big(\delta(x,t)\big)^{\frac{d+2}{d}}. 
\end{equation}
}
\end{remark}

%\todo[inline]{check whether $\rho_{\bfa}$ is equivalent with the parabolic distance function used by Aimar et al. for the special case of $W^{2,1}(\Omega)$...}

\section{Embeddings between anisotropic Kondratiev and Besov spaces}

\begin{theorem}[{\bf Embeddings between Kondratiev and Besov spaces}]\label{thm:emb-aniso}
%Let $\Omega=[0,T]\times D\subset \real^{d+1}$, where $D$ is a  Lipschitz domain of polyhedral type. 
Let $\bfm=m\bfa\in \nat^d$,  where the anisotropy $\bfa$ is normalized as in \eqref{anisotropy} and $\bfs=s\bfa, \ \bfr=r\bfa\in \rd_+$. 
Moreover, assume that the domain  $\Omega\subset \rd$ satisfies the strong $\bfr$-horn condition. 
Then we have a continuous embedding 
\begin{equation}\label{aniso-emb}
\mathcal{K}^{m\bfa}_{p,\gamma}(\Omega)\cap {B}^{s\bfa}_{p,p}(\Omega)\hookrightarrow {B}^{r\bfa}_{\tau,\tau}(\Omega), \qquad \frac{1}{\tau}=\frac rd+\frac 1p, \qquad 1<p<\infty,
\end{equation}
for all $0\leq r<\min(m, \frac{sd}{{d-1}})$ and $\gamma>\frac{\delta}{d}r$, where $\delta$ denotes the dimension of the singularity set $M\subset \partial \Omega$. 
%\todo{
%Kondratiev spaces are defined intrinsically in $\Omega$; only needed to estimate inner wavelet coefficients;\\
%Besov spaces are defined via restriction $\curvearrowright$ used in order to estimate boundary wavelet coefficients;
%}
\end{theorem}

\begin{proof}
Since for $r=0$ the result is clear, we assume in the sequel that $r>0$ and $0<\tau<p$. \\
{\em Step 1.} The proof is based on the wavelet characterization of Besov spaces from Definition \ref{def-B-aniso}. {Since our domain $\Omega$ satisfies the $\bfr$-horn condition (and thus the $\bfs$-horn condition), according to Corollary \ref{cor:ext-op} we can extend every $u\in B^{s\bfa}_{p,p}(\Omega)$ to some function $\tilde{u}=Eu\in B^{s\bfa}_{p,p}(\real^d)$. From this we deduce that in order to establish embedding \eqref{aniso-emb}   it is ultimately enough to show  
\begin{equation}\label{ineq-a}
\left(\sum_{(I,\psi)\in \Lambda}|I|^{\left(\frac 1p-\frac 12\right)\tau}|\langle \tilde{u},\tilde{\psi}_I\rangle|^{\tau}\right)^{1/\tau}\lesssim  \max \{\|u|\mathcal{K}^{m\bfa}_{p,\gamma}(\Omega)\|, \|u|B^{s\bfa}_{p,p}(\Omega)\|\}.  
\end{equation}
Let us give some further explanations here. We may extend the solution $u$ to a function $\tilde{u}= E(u)$ on the whole
Euclidean plane. Then, on $\Omega$, we have  $u=\sum_{(I,\psi)\in \Lambda} \langle \tilde{u}, \tilde{
\psi}_I\rangle \psi_I $. Therefore, if we can show that the expression on the right-hand
side is contained in the Besov space $B^{r\bfa}_{\tau,\tau}(\rd)$, the same is true for its
restriction to $\Omega$. To prove this, we use Definition \ref{def-B-aniso}. 
%Since the spaces $B^{r\bfa}_{\tau,\tau}(\Omega)$ are defined via restriction, the quasi-norm  $\|u|B^{r\bfa}_{\tau,\tau}(\Omega)\|$ is defined as an infimum over  possible representations of $u$ on $\rd$, which includes the wavelet decomposition of  $\tilde{u}=Eu\in B^{s\bfa}_{p,p}(\real^d)$ on  the left hand side in \eqref{ineq-a}. However, by comparison with the wavelet decomposition  \eqref{besov-decomp},  %:   Since our domain $\Omega$ satisfies the $\bfr$-horn condition (and thus the $\bfs$-horn condition), according to Theorem \ref{thm:ext-op} we can extend every $u\in B^{s\bfa}_{p,p}(\Omega)$ to some function $\tilde{u}=Eu\in B^{s\bfa}_{p,p}(\real^d)$. %%\todo{Alternatively: Since our Besov spaces are defined by restriction, we can find some function $\tilde{u}\in B^s_{p,p}(\real^d)$ such that the first term reads as...} 
Moreover,  we see that the first term there which reads as 
\[
\sum_{k\in \mathbb{Z}^d}\langle \tilde{u}, \tilde{\phi}(\cdot -k)\rangle {\phi}(\cdot - k)
\] 
(and also emerges in  \eqref{besov-norm}) does not appear on the left hand side of  \eqref{ineq-a}. % because it can be incorporated in the estimates that follow in Step 2. 
This is caused by the fact that ${\phi}$ shares the same smoothness and support properties as the wavelets $\psi_I$ for $|I|=1$ (note that below the vanishing moments of $\psi_I$  only become relevant for $|I|<1$). Therefore,  the coefficients $\langle \tilde{u}, \tilde{\phi}(\cdot -k)\rangle$ are incorporated in our considerations since they can be treated exactly like any of the coefficients $\langle \tilde{u},\tilde{\psi}_I\rangle$ in Step 2. %\todo[inline]{We only need  extension operator for anisotropic Besov spaces with $p>1$ if other  spaces with $\tau<1$ are defined via restriction!}
}\\
{\em Step 2.}  For our analysis we shall split the index set $\Lambda$ as follows.  For $j\in \nat_0$ the refinement level $j$ is denoted by %\todo{vorher: $ |I|=2^{-jd}$}
\[
\Lambda_j:=\{(I,\Psi)\in \Lambda: \ |I|=|\det M|^{-j}=\lambda^{-jd}\}.
\]
Furthermore, for $k\in \nat_0$ put 
\[
\Lambda_{j,k}:=\{(I,\psi)\in \Lambda_j: \ k\lambda^{-j}\leq \rho_{I,\bfa}<(k+1)\lambda^{-j}\},
\]
where 
$$\rho_{I,\bfa}=\inf_{x\in Q(I)} \rho_{\bfa}=\inf_{x\in Q(I), {y\in M}}|x-y|_{\bfa}.$$

%\todo[inline]{$\rho_{\bfa}(x):=\inf_{y\in M}|x-y|_{\bfa}$ denotes anisotropic distance to some singular set $M\subset \partial \Omega$, e.g. $M=\partial_{\text{par}}\Omega$}

%\todo[inline]{vorher: $\rho_I=\inf_{x\in Q(I)} \rho(x)$; $\rho_I$  replaced by anisotropic/parabolic distance} 

In particular, we have $\Lambda_j=\bigcup_{k=0}^{\infty}\Lambda_{j,k}$ and $\Lambda=\bigcup_{j=0}^{\infty}\Lambda_j$. \\
We consider first the situation $\rho_{I,\bfa}>0$ corresponding to  $k\geq 1$ and therefore put $\Lambda_j^0=\bigcup_{k\geq 1}\Lambda_{k,j}$. {Moreover, we require $Q(I)\subset \Omega$.} Recall the anisotropic version of Whitney's estimate  regarding approximation with polynomials from \cite[Lem.~2.1]{R12}, 
which states that for every $I$ there exists a polyomial $P_I\in \Pi_{M-1}$, where  $|\bfm|=m_1+\ldots+m_d\leq M-1$,  such that 
\begin{eqnarray*}
\|\tilde{u}-P_I|L_p(Q(I))\| &\lesssim &
\sum_{i=1}^d \lambda^{-j\frac{m_i}{a_i}}\|D^{m_i}_i \tilde{u}|L_p(Q(I))\|\\
&=&\lambda^{-jm} \sum_{i=1}^d \|D^{m_i}_i \tilde{u}|L_p(Q(I))\|
\lesssim  |I|^{m/d}  |\tilde{u}|_{W^{\bfm}_p(Q(I))}, 
%
%|Q(I)|^{m/d}|\tilde{u}|_{W^m_p(Q(I))}\leq c_1|I|^{m/d}|\tilde{u}|_{W^m_p(Q(I))}
\end{eqnarray*}
where the omitted constant is independent of $I$ and $u$. Here we used the fact that $\bfm=m\bfa$, i.e., $m=\frac{m_i}{a_i}$ for all $i=1,\ldots, d$,  and put 
%
%\[
%\|\tilde{u}-P_I|L_p(Q(I))\|\leq c_0 |Q(I)|^{m/d}|\tilde{u}|_{W^m_p(Q(I))}\leq c_1|I|^{m/d}|\tilde{u}|_{W^m_p(Q(I))}
%\]
%for some constant $c_1$ independent of $I$ and $u$, where 
$$|\tilde{u}|_{W^{\bfm}_p(Q(I))}:= \left({\sum_{i=1}^d}\int_{Q(I)}|D^{m_i}_i \tilde{u}(x)|^p\ud x\right)^{1/p}.$$ 
Note that $\tilde{\psi}_I$ can be chosen to satisfy moment conditions up to any  order, 
%e.g.,  $m\max\{a_1,\ldots, a_d\}$, 
we deduce that it is orthogonal to any polynomial $P_I\in \Pi_{M-1}$. Thus, using H\"older's inequality with $p>1$ we estimate 
\begin{eqnarray}
|\langle \tilde{u},\tilde{\psi}_I\rangle|
&=&|\langle \tilde{u}-P_I,\tilde{\psi}_I\rangle|\leq \|\tilde{u}-P_I|L_p(Q(I))\|\cdot \|\tilde{\psi}_I|L_{p'}(Q(I))\|\notag\\
&\lesssim &  |I|^{m/d}|\tilde{u}|_{W^{\bfm}_p(Q(I))}|I|^{\frac 12-\frac 1p}\notag\\
&\leq &  |I|^{\frac md+\frac 12-\frac 1p}\rho_{I,\bfa}^{\gamma-m}\left({\sum_{i=1}^d}\int_{Q(I)}|\rho_{\bfa}(x)|^{m-\gamma}{D^{m_i}_i}\tilde{u}(x)|^p \ud x\right)^{1/p}\notag\\
&=:&  |I|^{\frac md +\frac 12-\frac 1p}\rho_{I,\bfa}^{\gamma-m}\mu_{I,\bfa}. \label{est-hyper}
\end{eqnarray}
%{Note that in the third step above we require $\gamma<m$.} 
Note that in the third step we use that the values  of $\rho_{I,\bfa}$ and ${\rho}_{\bfa}$ are comparable, i.e., $\rho_{I,\bfa}\sim \sup_{x\in Q(I)} \rho_{\bfa}$, since  for  $k\geq 1$ we consider {cuboids} which do not intersect with the boundary. 
On the refinement level $j$, using H\"older's inequality with $\frac{p}{\tau}>1$, we find
\begin{align*}
\sum_{(I,\psi)\in \Lambda^0_j}&|I|^{\left(\frac 1p-\frac 12\right)\tau}|\langle \tilde{u},\tilde{\psi}_I\rangle|^{\tau}\\
&\leq  \sum_{(I,\psi)\in \Lambda_j^0}\left(|I|^{\frac md}\rho_{I,\bfa}^{\gamma-m}\mu_{I,\bfa}\right)^{\tau}\\
&\lesssim  \left(\sum_{(I,\psi)\in \Lambda_j^0}\left(|I|^{\frac md\tau}\rho_{I,\bfa}^{(\gamma-m)\tau}\right)^{\frac{p}{p-\tau}}\right)^{\frac{p-\tau}{p}}\left(\sum_{(I,\psi)\in \Lambda_j^0} \mu_{I,\bfa}^p\right)^{\tau/p}. 
\end{align*}

For the second factor we observe that there is a controlled overlap between the {cuboids} $Q(I)$, meaning each $x\in \Omega$ is contained in a finite number of {cuboids} independent of $x$, such that we get 
%\todo{$\alpha=\bfm$ anders schreiben!}
\begin{eqnarray*}
\left(\sum_{(I,\psi)\in \Lambda_j^0} \mu_{I,\bfa}^p\right)^{1/p}
&=& \left(\sum_{(I,\psi)\in \Lambda_j^0} {\sum_{i=1}^d}\int_{Q(I)}|\rho_{\bfa}^{m-\gamma}(x){D_i^{m_i}}\tilde{u}(x)|^p\ud x\right)^{1/p}\\
&\lesssim  &  \left({\sum_{i=1}^d}\int_{\Omega}|\rho_{\bfa}^{m-\gamma}(x){D_i^{m_i}}\tilde{u}(x)|^p\ud x\right)^{1/p}
\leq  \|u|\mathcal{K}^{m\bfa}_{p,\gamma}(\Omega)\|. 
\end{eqnarray*}
For the first factor,  by choice or $\rho_{\bfa}$  we always have $\rho_{I,\bfa}\leq  1$, hence the index $k$ is at most $\lambda^j$ for the sets $\Lambda_{j,k}$ to be non-empty. 
The number of elements in $\Lambda_{j,k}$  is bounded by $k^{d-1-\delta}\lambda^{j\delta}$. With this we find 
\begin{align*}
\Bigg(\sum_{(I,\psi)\in \Lambda^0_j}&\left(|I|^{\frac md\tau}\rho_{I,\bfa}^{(\gamma-m)\tau}\right)^{\frac{p}{p-\tau}}\Bigg)^{\frac{p-\tau}{p}}\\
&\leq  \left({\sum_{k=1}^{\lambda^j}\sum_{(I,\psi)\in \Lambda_{j,k}}}\left(\lambda^{-jm\tau}(k\lambda^{-j})^{(\gamma-m)\tau}\right)^{\frac{p}{p-\tau}}\right)^{\frac{p-\tau}{p}}\\
&\leq  \left(\sum_{k=1}^{\lambda^j}\sum_{(I,\psi)\in \Lambda_{j,k}}\left(\lambda^{-j\gamma\tau}k^{(\gamma-m)\tau}\right)^{\frac{p}{p-\tau}}\right)^{\frac{p-\tau}{p}}\\
&\lesssim   \left(
 \lambda^{-j\gamma\frac{p\tau}{p-\tau}}\sum_{k=1}^{\lambda^j}k^{(\gamma-m)\frac{p\tau}{p-\tau}}k^{d-1-\delta} \lambda^{j\delta}
\right)^{\frac{p-\tau}{p}}\\
&\lesssim  \lambda^{-j\gamma\tau} \lambda^{j\delta\frac{p-\tau}{p}}\left(\sum_{k=1}^{\lambda^j}k^{(\gamma-m)\frac{p\tau}{p-\tau}+d-1-\delta} 
\right)^{\frac{p-\tau}{p}}.\\
\end{align*}
Looking at the value of the exponent in the last sum we see that 
\[
(\gamma-m)\frac{p\tau}{p-\tau}+d-1-\delta>-1 \quad \iff \quad \gamma-m+r\frac{d-\delta}{d}>0,
\]
which leads to 
\begin{align}
\Bigg(\sum_{(I,\psi)\in \Lambda^0_j}& \left(|I|^{\frac md\tau}\rho_{I,\bfa}^{(\gamma-m)\tau}\right)^{\frac{p}{p-\tau}}\Bigg)^{\frac{p-\tau}{p}}\notag \\
&\lesssim  \lambda^{-j\gamma\tau}\lambda^{j\delta\frac{p-\tau}{p}}
\begin{cases}
\lambda^{j\left((\gamma-m)\tau+(d-\delta)\frac{p-\tau}{p}\right)}, & \gamma-m+r\frac{d-\delta}{d}>0,\\
(j+1)^{\frac{p-\tau}{p}}, & \gamma-m+r\frac{d-\delta}{d}=0, \label{cases-hyper}\\
1,& \gamma-m+r\frac{d-\delta}{d}<0.
\end{cases}
\end{align}
%The case  $\gamma>m$  can be treated in the same way as above by taking out $\tilde{\rho}_{I,\bfa}^{\gamma-m}$ with $\tilde{\rho}_{I,\bfa}:={\sup}_{x\in Q(I)}\rho_{\bfa}(x)$ instead of $\rho_{I,\bfa}^{\gamma-m}$ in the integral appearing in \eqref{est-hyper}. The values of $\rho_{I,\bfa}$ and $\tilde{\rho}_{I,\bfa}$ are comparable in this situation, since we consider cubes which do not intersect with the boundary, i.e., we have $k\geq 1$. In particular, in \eqref{cases-hyper}  only the first case occurs if $\gamma>m$. \\ 

%\textcolor{green}{Bis hierher gekommen!}
{\em Step 3.} We now put $\Lambda^0:=\bigcup_{j\geq 0}\Lambda_j^0$. Summing the first line of the last estimate over all $j$, we obtain 
\begin{align*}
\sum_{(I,\psi)\in \Lambda^0}&|I|^{\left(\frac 1p-\frac 12\right)\tau}|\langle \tilde{u}, \tilde{\psi}_I\rangle|^{\tau}\\
&\lesssim \sum_{j=0}^{\infty}\lambda^{-j(m\tau-d\frac{p-\tau}{p})}\|u|\mathcal{K}^{m\bfa}_{p,\gamma}(\Omega)\|^{\tau}\lesssim \|u|\mathcal{K}^{m\bfa}_{p,\gamma}(\Omega)\|^{\tau}<\infty,
\end{align*}
if the geometric series converges, which happens if 
\[
m\tau>d\frac{p-\tau}{p}\quad \iff\quad m>d\frac rd\quad  \iff \quad m>r.
\]
Similarly, in the second case we see that 
\begin{align*}
\sum_{(I,\psi)\in \Lambda^0}&|I|^{\left(\frac 1p-\frac 12\right)\tau}|\langle \tilde{u}, \tilde{\psi}_I\rangle|^{\tau}\\
&\lesssim  \sum_{j=0}^{\infty}\lambda^{-j(\gamma\tau-\delta\frac{p-\tau}{p})}(j+1)^{\frac{p-\tau}{p}}\|u|\mathcal{K}^{m\bfa}_{p,\gamma}(\Omega)\|^{\tau}\lesssim \|u|\mathcal{K}^{m\bfa}_{p,\gamma}(\Omega)\|^{\tau}<\infty,
\end{align*}
where the series converges if
\[
\gamma\tau>\delta \frac{p-\tau}{p},\quad \text{i.e.,} \quad \gamma>\delta \frac rd, \quad \text{i.e.,} \quad m>r\frac{d-\delta}{d}+\frac{\delta}{d}r=r,\quad \text{i.e.,} \quad m>r,
\]
which is the same condition as before. Finally, in the third case we find 
\begin{align*}
\sum_{(I,\psi)\in \Lambda^0}&|I|^{\left(\frac 1p-\frac 12\right)\tau}|\langle \tilde{u}, \tilde{\psi}_I\rangle|^{\tau}\\
&\lesssim  \sum_{j=0}^{\infty}\lambda^{-j(\gamma\tau-\delta \frac{p-\tau}{p})}\|u|\mathcal{K}^{m\bfa}_{p,\gamma}(\Omega)\|^{\tau}\lesssim \|u|\mathcal{K}^{m\bfa}_{p,\gamma}(\Omega)\|^{\tau}<\infty,
\end{align*}
whenever
\[
\gamma\tau>\delta\frac{p-\tau}{p}\quad \iff \quad \gamma>\delta\frac rd
\]
as in the second case above. \\
{\em Step 4.} We need to consider the sets $\Lambda_{j,0}$, i.e., the wavelets close to $M$. Here, we shall make use of the assumption $\tilde{u}\in B^{s\bfa}_{p,p}(\real^d)$. Since the number of elements in  $\Lambda_{j,0}$ is bounded from above by $c\lambda^{j\delta}$  we estimate using H\"older's inequality with $\frac{p}{\tau}>1$ and obtain 
\begin{align*}
\sum_{(I,\psi)\in \Lambda_{j,0}}&|I|^{\left(\frac 1p-\frac 12\right)\tau}|\langle \tilde{u}, \tilde{\psi}_I\rangle|^{\tau}\\
&\lesssim   %c_9^{\frac{p-\tau}{p}}
\lambda^{j\delta\frac{p-\tau}{p}}\left(\sum_{(I,\psi)\in \Lambda_{j,0}}\lambda^{-jd\left(\frac 1p-\frac 12\right)p}|\langle\tilde{u},\tilde{\psi}_I\rangle|^p\right)^{\tau/p}\\
&= %c_9^{\frac{p-\tau}{p}}
\lambda^{j\delta\frac{p-\tau}{p}}\lambda^{-js\tau}\left(\sum_{(I,\psi)\in \Lambda_{j,0}}\lambda^{j\left(s+\frac d2-\frac dp\right)p}|\langle\tilde{u},\tilde{\psi}_I\rangle|^p\right)^{\tau/p}.
\end{align*}
Summing up over $j$ and once more using H\"older's inequality with $\frac{p}{\tau}>1$ gives 
\begin{eqnarray*}
&\ds \sum_{j=0}^{\infty}&\sum_{(I,\psi)\in \Lambda_{j,0}}|I|^{\left(\frac 1p-\frac 12\right)\tau}|\langle \tilde{u}, \tilde{\psi}_I\rangle|^{\tau}\\
&\lesssim  & %c_9^{\frac{p-\tau}{p}}
\sum_{j=0}^{\infty}\lambda^{j\delta\frac{p-\tau}{p}}\lambda^{-js\tau}\left(\sum_{(I,\psi)\in \Lambda_{j,0}}\lambda^{j\left(s+\frac d2-\frac dp\right)p}|\langle\tilde{u},\tilde{\psi}_I\rangle|^p\right)^{\tau/p}\\
&\lesssim  & 
%c_9^{\frac{p-\tau}{p}}
\left(\sum_{j=0}^{\infty}\lambda^{j\delta }\lambda^{-js\tau\frac{p}{p-\tau}}\right)^{\frac{p-\tau}{p}}\cdot \left(\sum_{j=0}^{\infty}\sum_{(I,\psi)\in \Lambda_{j,0}}\lambda^{j\left(s+\frac d2-\frac dp\right)p}|\langle\tilde{u},\tilde{\psi}_I\rangle|\right)^{\tau/p}\\
&\lesssim  &\|\tilde{u}|B^{s\bfa}_{p,p}(\real^d)\|^{\tau}\lesssim \|u|B^{s\bfa}_{p,p}(\Omega)\|^{\tau},
\end{eqnarray*}
provided that %under the condition 
\[
\delta<\frac{sp\tau}{p-\tau}\quad \iff \quad \frac{s}{\delta}>\frac{1}{\tau}-\frac 1p=\frac rd \quad \iff \quad r<\frac{sd}{\delta}.
\]
{\em Step 5.} Finally, we need to consider those $\psi_I$ whose support intersect $\partial \Omega$. In this case we can estimate similar as in Step 4 with $\delta$ replaced by $d-1$. This results in the condition 
\[
\sum_{(I,\psi)\in \Lambda: \ \supp \psi_I\cap \partial \Omega\neq \emptyset}
|I|^{\left(\frac 1p-\frac 12\right)\tau}|\langle \tilde{u},\tilde{\psi}_I\rangle|^{\tau}
\lesssim \|\tilde{u}|B^{s\bfa}_{p,p}(\real^d)\|^{\tau}\lesssim \|u|B^{s\bfa}_{p,p}(\Omega)\|^{\tau}
\]
if $r<\frac{sd}{d-1}$. 
Altogether, we have proved 
\[
\|u|B^{r\bfa}_{\tau,\tau}(\Omega)\|\leq \|\tilde{u}|B^{r\bfa}_{\tau,\tau}(\real^d)\|
\lesssim \|u|B^{s\bfa}_{p,p}(\Omega)\|+\|u|\mathcal{K}^{m\bfa}_{p,\gamma}(\Omega)\|,
\]
with constants independent of $u$. 
 \end{proof}

\begin{remark}{
By a close inspection of  the proof of Theorem \ref{thm:emb-aniso} one sees that we have actually proven for any $u\in \mathcal{K}^{m\bfa}_{p,\gamma}(\Omega)\cap {B}^{s\bfa}_{p,p}(\Omega)$ that  
\begin{equation}\label{emb-inequ}
    \|u\|_{B^{r\bfa}_{\tau,\tau}(\Omega)}\lesssim \max\left\{|u|_{\mathcal{K}^{m \bfa}_{p,\gamma}(\Omega)}, \|u|B^{s\bfa}_{p,p}(\Omega)\|\right\},
\end{equation}
 where 
\begin{equation}\label{kondr-seminorm}
|u|_{\mathcal{K}^{m\bfa}_{p,\gamma}(\Omega)}:=\left({\sum_{i=1}^d}\int_{\Omega}\left|(\rho_{\bfa}(x))^{m-\gamma}D_i^{m_i}{u}(x)\right|^p \ud x\right)^{1/p} 
\end{equation}
denotes the Kondratiev semi-norm, where only the highest derivatives appear. 
}
\end{remark}

\section{Comparison and outlook: Anisotropic  regularity of the heat equation}

As already said before, we wish to study the regularity of parabolic problems (in particular, the heat equation) in anisotropic Besov spaces using the embedding from Theorem \ref{thm:emb-aniso} and compare our results with  \cite[Thms.~2]{AG12}. \\
Therefore, let the domain $\Omega=D\times [0,T]$ be a time-space cylinder, where $D\subset \real^{d+1}$ denotes a bounded Lipschitz domain, $M=\partial_{\mathrm{par}}\Omega
$ be the parabolic boundary which has dimension $\delta=d$, and consider the anisotropy $\bfa$ from \eqref{spec-aniso}, i.e., 
\[
\bfa=\left(a_1,\ldots, a_{d+1}\right)=\frac{d+2}{d} \left(1, \ldots, 1, \frac 12 \right). 
\]
 
Moreover, we denote by $\Theta(\Omega)$ the spaces of all temperatures 
\[
\Theta(\Omega):=\left\{u: \ \frac{\partial u}{\partial t}u=\Delta u \text{ in }\Omega\right\}. 
\]

Then the result from Aimar et al. obtained in \cite[Thms.~2]{AG12} reads as follows: 

\begin{theorem}
Let $1<p<\infty$, $\lambda>0$,  $\alpha>0$ and put $\frac{1}{\tau}=\frac 1p+\frac{\alpha}{d}$. Then 
%\begin{itemize}
%    \item[(i)] Let $l$ be the largest integer less than $\lambda+d$.  Then 
%    \begin{equation}\label{emb-aimar-1}
%    \Theta(\Omega)\cap L_p\big((0,T);B^{\lambda}_p(D)\big)\subset L_{\tau}\big((0,T); B^{\alpha}_{\tau}(D)\big),\quad \text{where}\quad \alpha <\min\left(l,\frac{\lambda d}{d-1}\right). 
%    \end{equation}
%    \item[(ii)] Moreover, 
    \begin{equation}\label{emb-aimar-2}
    \Theta(\Omega)\cap \mathbb{B}^{\lambda}_p(\Omega)\subset \bigcap_{\alpha>\varepsilon>0}\mathbb{B}^{\alpha-\varepsilon}_{\tau}(\Omega),\quad \text{where}\quad \alpha <\min\left(d\Big(1-\frac 1p \Big),\frac{\lambda d}{d-1}\right). 
    \end{equation}
%\end{itemize}
\end{theorem}

{
In particular, this  result was obtained with the help of gradient estimates of temperatures. In this context  we recall \cite[Thm.~5]{AG12}, which will be useful for us in the sequel. We make use of the following notation: we write $\nabla^{2,1}u$ to denote the $(d^2+1)$-vector given by the $d^2$ second-order purely spatial derivatives of $u$ and the first derivative of $u$ w.r.t. time, i.e., $\nabla^{2,1}u=\left(\nabla^2 u, \frac{\partial u}{\partial t}\right)$. By $(\nabla^{2,1})^nu$, $n\in \nat$, we denote the vector of all derivatives, where  each component has the form $\partial^{(\alpha,\alpha_{d+1})}u$ with $|\alpha|+2\alpha_{d+1}=2n$. This way we always have in each one of these derivatives an even number of space derivatives. Moreover, $|(\nabla^{2,1})^nu|$  denotes the Euclidean length of $(\nabla^{2,1})^nu$. Then \cite[Thm.~5]{AG12}, adapted to our situation, reads as follows. 
}

{
\begin{corollary}\label{cor:grad-est} 
Let $\Omega=D\times [0,T]$ with $D\subset \rd$ be a bounded Lipschitz domain, $\lambda>0$, $n\in \nat$,   and $1<p<\infty$.  Then there exists a constant $c$ depending on    $d$, $\lambda$, $p$, and the Lipschitz character of $D$ such  that 
\begin{equation}\label{est-gradient}
\big\|\delta^{2n-\lambda}|(\nabla^{2,1})^{n}  u|  \big|L_p(\Omega)\big\| \leq c \left\|u|L_p([0,T], {B}^{\lambda}_{p,p}(D))\right\| {\leq c'  \|u|\mathbb{B}^{\lambda}_p(\Omega)\|}
\end{equation}
holds for every  temperature $u$ in $\Theta(\Omega)$. 
\end{corollary}
\remark{Corollary \ref{cor:grad-est} in its above version is a consequence of \cite[Thms.~4,5]{AG12}  together with \cite[Lem.~5.3]{AGI08}, where the latter is the observation that the derivative $(\nabla^{2,1})^{n}u$ belongs to the linear  span of $\nabla^{2n}u$ . 
}
}

%In particular, this  result was obtained with the help of gradient estimates of temperatures. In this context  we recall \cite[Cor.~6.2]{AGI08} {adapted to our situation,}  which will be useful for us in the sequel. 
%\todo[inline]{Change from time interval $\real_+$ to $[0,T]$!}

%\begin{corollary}\label{cor:grad-est} 
%If $\Omega=D\times [0,T]$ with $D\subset \rd$ a bounded Lipschitz domain, $\lambda\in \real_+\setminus \nat$ and $n\in \nat$ such that $\lfloor\lambda\rfloor +1\leq 2n<\lambda+d$. Then 
%\begin{equation}\label{est-gradient}
%\left\|\delta^{2n-\lambda}|(\nabla^{2,1})^n u|\big|L_p(\Omega)\right\|\leq c \left\|u|L_p(\real_+, {B}^{\lambda}_{p,p}(D))\right\| {\leq c'  \|u|\mathbb{B}^{\lambda}_p(\Omega)\|}
%\end{equation}
%for  constants $c$, $c'$ and every temperature $u$ in $\Omega$. 
%\end{corollary}

We can reinterpret the estimate \eqref{est-gradient}  in terms of anisotropic Kondratiev regularity for the homogeneous heat equation as follows:  
The left hand side in \eqref{est-gradient} can be expressed via the Kondratiev semi-norm \eqref{kondr-seminorm}, since  using \eqref{comp-weights} we see that for $\bfm=(2n,\ldots, 2n,n)=2n\frac{d}{d+2}\bfa=:m\bfa$ and $s:=\frac{d+2}{d}\gamma=\lambda$  we have 
\begin{align*}
    |u|_{\mathcal{K}^{\bfm}_{p,\gamma}(\Omega)}& \sim {\sum_{i=1}^d}\left\|(\rho_{\bfa})^{m-\gamma}{D^{m_i}_i}{u}\big| L_p(\Omega)\right\|  \\
     & \sim {\sum_{i=1}^d}\left\|{\left(\delta^{\frac{d+2}{d}}\right)}^{m-\gamma}{D^{m_i}_i}{u}\big| L_p(\Omega)\right\| \\ 
    & = {\sum_{i=1}^d}\left\|\delta^{2n-s}{D^{m_i}_i}{u}\big| L_p(\Omega)\right\|\lesssim \left\|\delta^{2n-\lambda}|(\nabla^{2,1})^n u|\big|L_p(\Omega)\right\|.  
\end{align*}
Thus, a combination of Theorem \ref{thm:emb-aniso}, Corollary \ref{cor:grad-est}, and the observation that $B^{\tilde{s}\bfa}_{p,p}(\Omega)=\mathbb{B}^s_p(\Omega)$ for $\tilde{s}={s}\frac{d}{d+2}$   
yields  for a temperature $u\in \Theta(\Omega)$ : 
\begin{align}
    \|u|B^{r\bfa}_{\tau,\tau}(\Omega)\|\lesssim \max\left\{|u|_{\mathcal{K}^{\bfm}_{p,\tilde{s}}(\Omega)}, \|u|B^{\tilde{s}\bfa}_{p,p}(\Omega)\|\right\} \lesssim \|u|B^{\tilde{s}\bfa}_{p,p}(\Omega)\|
\end{align}
subject to the restriction 
\[
0<r<\min\left(2n \frac{d}{d+2}, \tilde{s}\frac{d+1}{d}\right). 
\]
%and 
%\begin{equation}\label{rest-2n}
%\lfloor {s}\rfloor +1\leq 2n <{s}+d.
%\end{equation}
In good agreement with \eqref{coinc-B-spaces} we put  $B^{r\bfa}_{\tau,\tau}(\Omega):=\mathbb{B}^{\alpha}_{\tau}(\Omega)$ for $\alpha=r\frac{d+2}{d}$ (i.e.,  the space $\mathbb{B}^{\alpha}_{\tau}(\Omega)$ with $\tau<1$ has to be understood -- in a slight abuse of notation --  according to Definition \ref{def-B-aniso}) and we obtain: 
\begin{equation}\label{emb-aimar-improved}
    \Theta(\Omega)\cap \mathbb{B}^{s}_p(\Omega)\subset \mathbb{B}^{\alpha}_{\tau}(\Omega),\quad \text{where}\quad \alpha <\min\left(2n,s\frac{d+1}{d}\right). 
    \end{equation}
%with  $2n$ as in  \eqref{rest-2n}. 
%By choosing $2n$ as large as possible ($2n\geq \lfloor {s}\rfloor+d-1$ is always possible), 
%Since we can choose $2n >\tilde{s}+d-1$ our result  improves \eqref{emb-aimar-2} if
%$s\frac{d}{d-1}<\tilde{s}+d-1$, which happens if, and only, if 
%\[
%s<\frac{d-1}{\frac{1}{d+1}+\frac{2}{d+2}}. 
%\]
%In particular, we obtain an improvement  for $s<\frac 65$ if $d=2$ and $s<\frac{40}{13}$ if $d=3$.  
 Comparing \eqref{emb-aimar-improved} with \eqref{emb-aimar-2} we conclude that the restriction on the smoothness parameter $\alpha$ has improved significantly: Since $\Omega\subset \real^{d+1}$ by replacing $d$ by $d+1$ our  approach  gives a much better upper bound for $\alpha$ \Big($s\frac{d+1}{d}$ instead of $s\frac{d}{d-1}$\Big).   
Moreover, the restriction  $\alpha<d\Big(1-\frac 1p \Big)$ resulting from the fact that the spaces $\mathbb{B}^{\lambda}_{p}(\Omega)$ defined in \cite{AG12} only make sense for $p>1$ can be completely removed. \\
In particlar, invoking   \cite[Thm.~6.2]{Wo07} we deduce  that 
\[
\Theta(\Omega)\subset 
   W^{2\ldots, 2,1}(\Omega)= B^{\frac{2d}{d+2}\bfa}_{2,2}(\Omega)=\mathbb{B}^2_2(\Omega). 
\]
i.e., \eqref{emb-aimar-improved} yields for parameters  $p=2$ and  $s<2$ {(assuming $2n$ large)} that 
%\todo[inline]{wir m\"ussen  $s<2$ waehlen, da $s=\lambda$ nicht ganzzahlig sein soll in Corollary \ref{cor:grad-est}}
\[
\Theta(\Omega)\subset \mathbb{B}^{\alpha}_{\tau}(\Omega), \qquad \text{where}\quad \alpha <\frac 83 \ (d=3)\quad \text{and}\quad \alpha <{3} \ (d=2). 
\]
On the other hand \eqref{emb-aimar-2} only yields $\alpha<\frac 32 \ (d=3)$ and $\alpha<1\ (d=2)$.

\begin{remark}{
Let us note that compared to \cite{AG12} our approach is more flexible: It  allows us to treat more general parabolic equations (also with inhomogeneous initial boundary data) as long as one has regularity results for the solution of the parabolic problem in anisotropic Kondratiev spaces. In this context we mention \cite{KR12} for first results in this direction. \\
Moreover, we think that even better results can be achieved if  one investigates regularity in anisotropic Kondratiev and Besov spaces which have different integrability w.r.t. the spacial  and time variable.  
This interesting problem will be studied a future paper. \\
Finally, we remark that it is not completely clear that the anisotropic Kondratiev and Besov spaces we are dealing with in this paper are the optimal spaces for  studying parabolic PDEs. Another possibility would be to have a look at the regularity of the solutions to evolution equations  in Besov spaces of dominating mixed smoothness  type  or even Banach-valued Besov spaces. 
}
\end{remark}

\end{document}